\documentclass{amsart}
\usepackage{amsfonts}
\usepackage{amsmath,amssymb}
\usepackage{amsthm}
\usepackage{amscd}
\usepackage{graphics}
\usepackage{graphicx}

\theoremstyle{remark}{

\newtheorem{Rem}{{\rm Remark}}
\newtheorem{Prob}{{\rm Problem}}

}
\theoremstyle{plain}
{

\newtheorem{Thm}{Theorem}

}

\begin{document}
\title[Morse functions fibered by spheres, tori, or Klein Bottles]{Morse functions with regular level sets consisting of $2$-dimensional spheres, $2$-dimensional tori, or Klein Bottles}
\author{Naoki kitazawa}
\address{Osaka Central Advanced Mathematical Institute (OCAMI) \\
3-3-138 Sugimoto, Sumiyoshi-ku Osaka 558-8585
TEL: +81-6-6605-3103}
\email{naokikitazawa.formath@gmail.com}
\urladdr{https://naokikitazawa.github.io/NaokiKitazawa.html}
\maketitle
\begin{abstract}
In this paper, we study Morse functions with regular level sets consisting of spheres, tori, or Klein Bottles on $3$-dimensional closed manifolds.

We characterize $3$-dimensional manifolds represented by connected sums each of whose summands is the product $S^1 \times S^2$ of the circle $S^1$ and the sphere $S^2$, lens spaces, or non-orientable closed and connected manifolds of genus $1$ by a certain subclass of such Morse functions. This is a kind of extensions of the orientable case, by Saeki, in 2006. This is a variant of its extension by the author for $3$-dimensional orientable manifolds represented by connected sums each of whose summands is the product $S^1 \times S^2$, lens spaces, or torus bundles over $S^1$ by a certain class of Morse-Bott functions.
We also classify Morse functions with given regular level sets consisting of $S^2$, $S^1 \times S^1$, or Klein Bottles in a certain sense, generalizing some previous work by the author.

\end{abstract}
%【REVISE】 combinatoric ～ is → combinatorial object. It is .
%【REVISE】  such that a point is a vertex if and only if the corresponding connected component of the level set contains some singular points → whose vertex set is the set of all points containing some singular points in the corresponding connected component of the level set .
%【REVISE】 We delete "extending the result before".
\section{Introduction.}
\label{sec:1}
{\it Morse} functions have been fundamental and strong tools in investigating the manifolds as objects in differential topology, geometric topology, and various geometry and mathematics.
Critical points of Morse functions appear discretely and they have information on homology groups of the manifolds, decomposition of the manifolds into disks ({\it handles}), and information of homotopy of the manifolds.
%This is a fundamental theorem in theory of Morse functions. See 
See \cite{milnor1, milnor2} for this fundamental theory. This is extended to and applied in infinite dimensional situations and in \cite{milnor1}, classical theory is presented, and see also related study \cite{palais} for example.  In low dimensional geometry (differential topology or geometric topology), certain diagrams (such as {\it Kirby diagrams} in low dimensional manifolds), based on critical points of Morse functions and handles are important.

In this paper, we emphasize fundamental philosophy that Morse functions are not only tools, but also important objects in various geometry. Our study is also regarded as a topic from singularity theory of differentiable maps and applications to differential topology and geometric topology of manifolds.

Note that some of the present exposition is presented based on a slide for a presentation and a report of the author \cite{kitazawa3, kitazawa4}, in "Mathematical Science of Knots VIII", a conference on knot theory and related mathematics. In the present paper, we use some terminologies and notions such as {\it critical} points (, the {\it critical set} and the {\it critical value set}) of a real-valued smooth function $c:X \rightarrow \mathbb{R}$, a {\it Morse} function, {\it singular} points (, the {\it singular set} and the {\it singular value set}) of a smooth map $c:X \rightarrow Y$ between smooth manifolds, and a graph with related notions, with no precise exposition. For a smooth real-valued function $c:X \rightarrow \mathbb{R}$, a preimage $f^{-1}{(r)}$ is a {\it level set} and it is {\it regular} if it contains no critical point (of the function $c$).
 We use $D^k \subset {\mathbb{R}}^k$ for the $k$-dimensional (unit) disk in the $k$-dimensional Euclidean space ${\mathbb{R}}^k$ and the boundary is the ($k-1$)-dimensional (unit) sphere $\partial D^k=S^{k-1}$. A connected sum of manifolds can be the sphere $S^m$, where there is no connected summand. We use $K^2$ for the Klein Bottle. 
For systematic understanding of $3$-dimensional manifold theory, refer to \cite{hempel} for example.

In differential topology and geometric topology, characterizations of certain manifolds by the existence of Morse functions of certain classes are important, as Reeb's sphere theorem implies (see \cite{reeb} and see also \cite{milnor2} again). Theorems of this type have been presented in the development of global singularity theory of differentiable maps and related differential topology and geometric topology, mainly due to Saeki: \cite{saeki1, saeki2} are of pioneering related studies. For classical studies on higher dimensional versions of Morse functions, \cite{thom, whitney} are important, and as another related study, studies on existence of {\it fold} maps into ${\mathbb{R}}^n$, higher dimensional versions of Morse functions, via theory of differential equations and so-called homotopy principle, are important and known as celebrated theory by Eliashberg (\cite{eliashberg1, eliashberg2}).

\begin{Thm}[\cite{saeki4}]
\label{thm:1}
A $3$-dimensional closed, connected, and orientable manifold $M$ admits a Morse function $f:M \rightarrow \mathbb{R}$ such that regular level sets $f^{-1}(r)$ consist of surfaces diffeomorphic to $S^2$ or $S^1 \times S^1$ if and only if $M$ is diffeomorphic to a connected sum each of whose summand is $S^1 \times S^2$ or a so-called {\rm lens space}.
\end{Thm}
For this, fundamental $3$-manifold theory and the well-known correspondence between critical points of Morse functions and handles are important.  
A {\it Morse-Bott} function is a kind of generalizations of a Morse function. See \cite{banyagahurtubise} and see also \cite{bott}. Theorem \ref{thm:2} is a generalization of Theorem \ref{thm:1}.
\begin{Thm}[\cite{kitazawa8}]
\label{thm:2}
A $3$-dimensional closed, connected, and orientable manifold $M$ admits a Morse-Bott function $f:M \rightarrow \mathbb{R}$ such that regular level sets $f^{-1}(r)$ consist of surfaces diffeomorphic to $S^2$ or $S^1 \times S^1$and the following are satisfied if and only if $M$ is diffeomorphic to a connected sum each of whose summands is $S^1 \times S^2$, a lens space, or a so-called {\rm torus bundle} over $S^1$.
\begin{itemize}
\item Around each critical point where $f$ has no local extremum, $f$ is Morse.
\item The critical set of $f$ is diffeomorphic to a disjoint union of manifolds each of which is a single point, or diffeomorphic to the circle $S^1$, the sphere $S^2$, the torus $S^1 \times S^1$, or the real projective plane ${\mathbb{R}P}^2$.
\end{itemize}
\end{Thm}
For this, {\it Reeb} ({\t di}){\it graphs} of Morse-Bott functions with data on regular level sets and deformations are important. We present {\it Reeb} ({\it di}){\it graphs} of smooth functions (\cite{reeb}). For a smooth function $c:X \rightarrow \mathbb{R}$ on a manifold $X$ with no boundary, the quotient space $R_c:=X/{\sim}_c$ is defined by the following equivalence relation ${\sim}_c$ on $X$. For two points $x_1,x_2 \in X$, $x_1 {\sim}_c x_2$ if and only if they are in a same connected component of a same level set $c^{-1}(y)$. This is the {\it Reeb space} of $c$. It has the structure of a graph in specific cases. We can define the quotient space $q_c:X \rightarrow R_c$ and the unique continuous function $\bar{c}:R_c \rightarrow \mathbb{R}$ with $c=\bar{c} \circ q_c$.
For the Morse(-Bott) function case, see \cite{izar} (resp. \cite{martinezalfaromezasarmientooliveira}). For a general situation, see \cite{saeki5, saeki6} (\cite[Theorem 3.1]{saeki5}) and more rigorously $R_c$ is a graph whose vertex set consists of all points $v$ such that ${q_c}^{-1}(v)$ has some critical points of $c$. We make $R_c$ a digraph by giving the orientation of each edge by $\bar{c}$ with the rule that an edge $e$, incident to exactly two vertices $v_{e,1}$ and $v_{e,2}$, is oriented as an edge departing from $v_{e,1}$ and entering $v_{e,2}$ if and only if $\bar{c}(v_{e,1})<\bar{c}(v_{e,2})$. We can define a digraph as a pair $(G,c_G)$ of a finite and connected graph $G$ and a continuous map $c_G$ which is injective on each edge, oriented according to the rule above by using $c_G$. We omit $c_G$ for this unless we need. We can define {\it isomorphisms between digraphs} as isomorphisms of graphs preserving the orders of the functions.
For digraphs, a {\it sink} ({\it source}) means a vertex which every edge incident to it enters (from which every edge incident to it departs). Theorem \ref{thm:3} is a kind of fundamental propositions. 
\begin{Thm}
\label{thm:3}
For a Morse function on a closed and connected manifold $c$, $R_c$ is a finite and connected digraph whose source and sinks are of degree $1$. We also have a pair $(R_c,\{{q_f}^{-1}(p_e)\}_{e \in E_{R_c}})$, where $E_{R_c}$ denotes the edge set, with $p_e$ being a point in {\rm (}the interior of{\rm )} an edge $e \in E_{R_c}$. We call this the {\rm Reeb data} of $c$.
\end{Thm}
For this, see also \cite{marzantowiczmichalak, michalak1, michalak2}. Notions (and arguments essentially) presented first by the author, in \cite{kitazawa2}, are important and we explain them. Let $G$ be a finite and connected digraph such that the restriction of $c_G$ to each edge is injective and that sources and sinks are always of degree $1$. We also consider a family $\{F_e\}_{e \in E_G}$ of ($m-1$)-dimensional closed and connected manifolds labeled by the edge set $E_G$ such that for each edge $e_{\rm s}$ incident to some sink or source, $F_{e_{\rm s}}=S^{m-1}$ and call $(G,\{F_{e}\}_{e \in E_G})$ an ({\it $m-1$}){\it -labeled-pre-M-digraph}. If we can have such an object as in Theorem \ref{thm:3} up to isomorphisms, then it is called an ({\it $m-1$}){\it -labeled-M-digraph}. Here, an {\it isomorphism} means an isomorphism of the digraphs mapping each edge $e_1$ to another edge $e_2$ in such a way that $F_{e_1}$ and $F_{e_2}$ are diffeomorphic.

Theorem \ref{thm:4} is one of our main result. Here, a ($3$-dimensional) manifold of {\it degree $g \geq 0$} means a closed manifold obtained by gluing two copies of ${\natural}_{j=1}^g (D^2 \times S^1)$ or ${\natural}_{j=1}^g (S^1 \tilde{\times} D^2)$ along the boundaries, where $g$ cannot be smaller. A {\it lens space} is an orientable manifold with $g=1$ which is not homeomorphic to $S^1 \times S^2$ here. Here, the notation is used in the following way: ${\natural}_{j=1}^l X_j$ is for a so-called {\it boundary sum} of $j$ manifolds $X_j$ and
$S^1 \tilde{\times} D^2$ is for the total space of a non-trivial smooth bundle over the circle $S^1$ whose fiber is the $2$-dimensional disk $D^2$.
\begin{Thm}
\label{thm:4}
\begin{enumerate}
\item A $3$-dimensional closed and connected manifold $M$ admits a Morse function $f:M \rightarrow \mathbb{R}$ with the following if and only if $M$ is diffeomorphic to a connected sum each of whose summand is diffeomorphic to $S^1 \times S^2$, a lens space, or a non-orientable manifold of degree $1$.
\begin{enumerate}
\item \label{thm:4.1} For the Reeb data $(R_c,\{{q_f}^{-1}(p_e)\})$, ${q_f}^{-1}(p_e)$ is diffeomorphic to $S^2$, $S^1 \times S^1$, or $K^2$.
\item \label{thm:4.2} For each vertex $v$, either an edge $e_{{\rm  e},v}$ entering $v$ with ${q_f}^{-1}(p_e)$ diffeomorphic to $K^2$ or an edge $e_{{\rm  d},v}$ departing from $v$ with ${q_f}^{-1}(p_e)$ diffeomorphic to $K^2$, exists.
 \end{enumerate}
\item
A $3$-dimensional closed and connected manifold $M$ whose 2nd Stiefel-Whitney class vanishes admits a Morse function $f:M \rightarrow \mathbb{R}$ satisfying {\rm (}\ref{thm:4.1}{\rm )}, presented above, if and only if $M$ is diffeomorphic to a connected sum each of whose summands is diffeomorphic to $S^1 \times S^2$, a lens space, or a non-orientable manifold of degree 1.
\end{enumerate}
\end{Thm}

The content of the present paper is as follows. In the next section, we prove Theorem \ref{thm:4}. In the third section, as another work, we discuss classifications of Morse functions on a fixed $m$-dimensional manifold up to isomorphisms of ({\it $m-1$})-labeled-pre-digraphs (Theorem \ref{thm:6}). This is a kind of natural classifications weaker than the classifications up to so-called {\it topological} ({\it $C^{\infty}$}){\it -equivalences}.
\section{Characterizations of $3$-dimensional manifolds of certain classes and one of our main result, Theorem \ref{thm:4}.}

A Morse function is {\it simple} if at distinct critical points of the function the values are always mutually distinct.
Let $f:M \rightarrow N$ be a smooth map from an $m$-dimensional manifold $M$ into an $n$-dimensional one $N$ with no boundary where $m \geq n \geq 1$. $f$ is a {\it fold} map if at each singular point $p$ of $f$, for a suitable local coordinates and the uniquely defined integer $0 \leq i(p) \leq \frac{m-n+1}{2}$, $f(x_1,\cdots x_m)=(x_1 \cdots x_{n-1}, {\Sigma}_{j=1}^{m-n-i(p)+1} {x_{n-1+j}}^2-{\Sigma}_{j=1}^{i(p)} {x_{m-i(p)+j}}^2)$. This is explained as a smooth map locally represented as a projection (with no singular point) or the product map of a Morse function and the identity map on ($n-1$)-dimensional manifold with no boundary. For this see \cite{eliashberg1, eliashberg2, saeki1} for example and for systematic exposition on fundamental notions and arguments on singularity theory of differentiable maps, see \cite{golubitskyguillemin} as one of related well-known textbooks.

We explain the fundamental correspondence between {\it $j$-handles} and critical points of {\it index} $j$ shortly.
We consider two spaces $F_a$ and $F_b$ each of which is an ($m-1$ )-dimensional smooth closed manifold or empty.  
 Let $f_{F_a,F_b}:\tilde{M_{F_a,F_b}} \rightarrow \{t \mid a \leq t \leq b\}$ a Morse function on an $m$-dimensional compact and connected manifold $\tilde{M_{F_a,F_b}} $ with the following.
\begin{itemize}
\item For the boundary $\partial \tilde{M_{F_a,F_b}}$ of $\tilde{M_{F_a,F_b}}$, it holds that $\partial \tilde{M_{F_a,F_b}}=F_a \sqcup F_b$.
\item It has the unique critical value $s$ with $a<s<b$.
\item The level set ${f_{F_a,F_b}}^{-1}(a)$ is $F_a$ and the level set ${f_{F_a,F_b}}^{-1}(b)$ is $F_b$.
\end{itemize}
We explain the topology and the differentiable structure of $\tilde{M_{F_a,F_b}}$. We prepare $F_a \times \{t \mid 0 \leq t \leq 1\}$, where $F_a$ and $F_a \times \{0\}$ are identified by the map mapping $x$ to $(x,0)$. By attaching $j$-handles $D^j \times D^{m-j}$ to $F_a \times \{1\}$ by diffeomorphisms from $\partial (D^j) \times D^{m-j}$ disjointly and eliminating the resulting corner canonically, we have $\tilde{M_{F_a,F_b}}$. Each $j$-handle corresponds to a critical point of {\it index} $j$ of the function, where the {\it index} is defined as an integer $0 \leq j \leq m$ respecting the order of the values of the function.

We prove Theorem  \ref{thm:4} (\ref{thm:4.1}). We respect arguments of the preprints  \cite{kitazawa7, kitazawa8} for this and other new result of us, where we discuss in a self-contained way. It is a new part that we consider the situation with some edge $e$ satisfying $F_e=K^2$.

Hereafter, for example, a so-called {\it path digraph on $l \geq 2$ vertices $v_j$} ($1 \leq j \leq l$) is important. This is also a digraph with the edge set $\{e_{j}\}_{j=1}^{l-1}$ each element $e_j$ of which is oriented as a directed edge departing from $v_j$ and entering $v_{j+1}$. In addition, in our situation, $F_{e_j}=S^2$ for $j=1,l-1$. 

\begin{proof}[A proof of Theorem \ref{thm:4} (\ref{thm:4.1})]
Step 1 A proof of "The "If" part". \\

For $S^3$, we consider a Morse function with exactly two critical points. Its Reeb data is isomorphic to $(\{e_0\},\{R_{e_0}=S^2\})$. For $S^1 \times S^2$, a len space, or a non-orientable manifold of degree $1$, we can consider a simple Morse function whose Reeb data is isomorphic to $(\{e_1,e_2,e_3\},\{R_{e_1}=S^2,R_{e_2}=S^1 \times S^1,K^2, R_{e_3}=S^2\})$ and the digraph is a path digraph on $4$ vertices. See FIGURE \ref{fig:1}.
\begin{figure}
	\includegraphics[width=70mm,height=65mm]{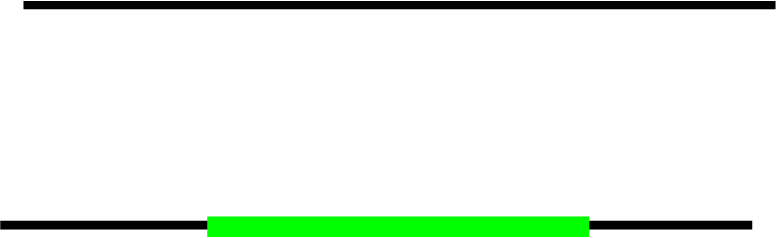}
	\caption{Fundamental Morse functions in STEP 1. Their Reeb data are presented roughly: for the black colored edges $e$, $F_e=S^2$, and for the green colored edge $e$, $F_e=S^1 \times S^1,K^2$.}
	\label{fig:1}
\end{figure}
We can have a desired Morse function by iterations of fundamental operations, presented in FIGURE \ref{fig:2}. 
\begin{figure}
	\includegraphics[width=70mm,height=65mm]{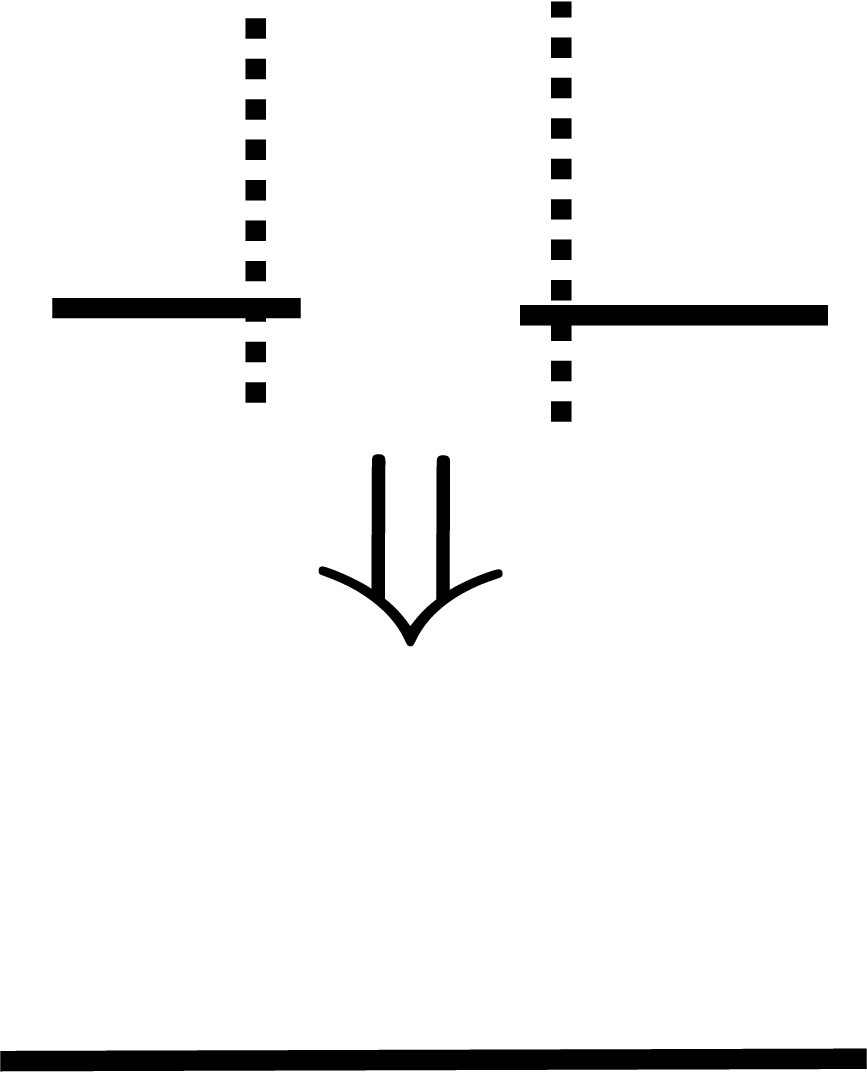}
	\caption{Construction of a new Morse function on a connected sum of given two manifolds, with given Morse functions, in STEP 1, where the Reeb digraphs are presented locally.}
	\label{fig:2}
\end{figure}

\ \\
Step 2 A proof of "The "Only if" part". \\
We first investigate the topology (differentiable structure) of the manifold and the function mapped to a small neighborhood of each vertex of the Reeb digraph. We can also deform the local Morse function to a simple Morse function. After this deformation, we deform the function further and investigate the global topology (differentiable structure) of the manifold. \\
\ \\
Case 2-1 Around a source or a sink $v_1$. \\
Around such a vertex, we have the natural height of the disk $D^3$, represented by the form $f(x)=\pm ({x_1}^2+{x_2}^2+{x_3}^2)+\bar{f}(v_1)$, where the isolated critical point $(x_1,x_2,x_3)=(0,0,0)$ is mapped to $v_1$ by the quotient map $q_f$. \\
\ \\
Case 2-2 Around a vertex $v_2$ of degree $2$ such that for the two edges $e_{v_2,1}$ and $e_{v_2,2}$ incident to $v_2$, $F_{e_{v_2,1}}$ and $F_{e_{v_2,2}}$ are diffeomorphic to $S^2$. \\

By the observation of the handle attachment, we have the following. Around such a vertex $v_2$, we have a $3$-dimensional compact and connected submanifold of $M$ bounded by $F_{e_{v_2,1}} \sqcup F_{e_{v_2,2}}$ and obtained by attaching $l_{v_2}>0$ $2$-handles to $F_{e_{v_2,1}} \times \{1\} \subset F_{e_{v_2,1}} \times \{t \mid 0 \leq t \leq 1\}$ and $l_{v_2}>0$ $1$-handles to $F_{e_{v_2,1}} \times \{1\} \subset F_{e_{v_2,1}} \times \{t \mid 0 \leq t \leq 1\}$ disjointly, where  $F_{e_{v_2,1}}$ is identified with  $F_{e_{v_2,1}} \times \{0\}$ by the map mapping $x$ to $(x,0)$. 
 \begin{figure}
	\includegraphics[width=70mm,height=65mm]{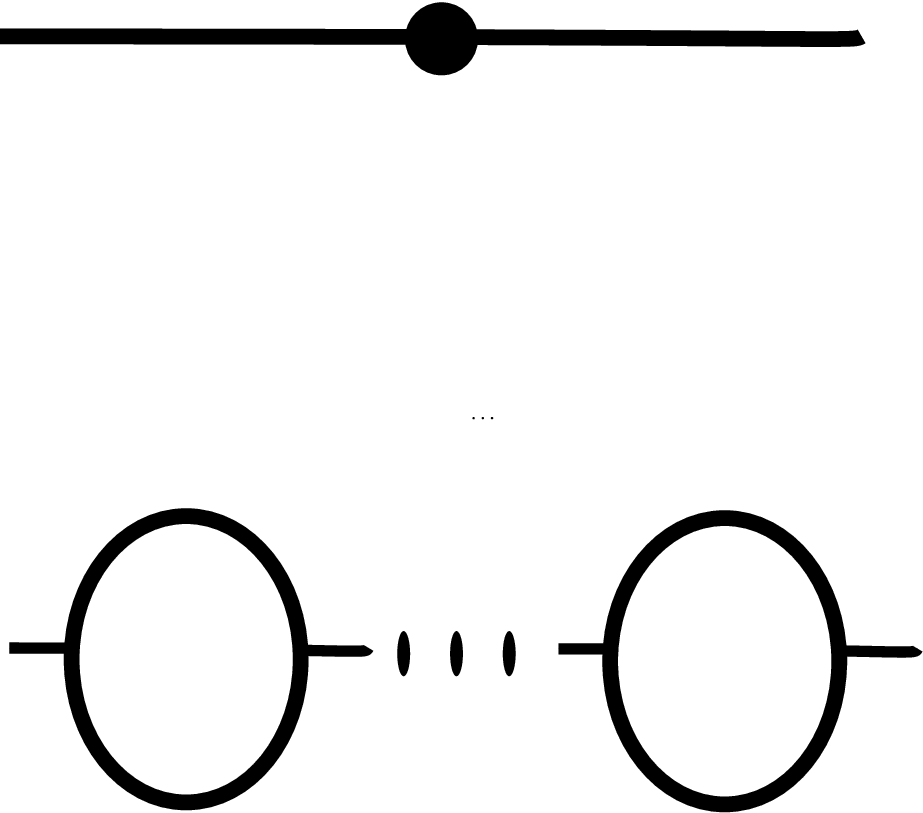}
	\caption{Deforming a function in Case 2-2 to a local simple Morse function in such a way that the 1st Betti number of the Reeb graph increases by $l_{v_2}$ and the Reeb digraphs are presented locally and the regular level sets of the resulting local Morse function are diffeomorphic to $S^2$ or $S^2 \sqcup S^2$.}
	\label{fig:3}
\end{figure}

\ \\
Case 2-3 Around remaining vertices $v$. \\
We consider the family $\{e_{{\rm e},v,j}\}_{j=1}^{l_{\rm e}}$ of all edges entering $v$ and the family $\{e_{{\rm d},v,j}\}_{j=1}^{l_{\rm d}}$ of all edges departing from $v$. 

Suppose that ${\sqcup}_{j=1}^{l_{\rm e}} F_{e_{{\rm d},v,j}}$ contains no connected component diffeomorphic to $K^2$. In this case, for the surface ${\sqcup}_{j=1}^{l_{\rm e}} F_{e_{{\rm e},v,j}}$ we embed into ${\sqcup}_{j=1}^{l_{\rm e}} (F_{e_{{\rm e},v,j}}) \times \{t \mid 0 \leq t \leq 1\}$ by the map mapping $x$ to $(x,0)$ first, and attach all $2$-handles corresponding to the critical points of index $2$ for the local Morse function disjointly, to ${\sqcup}_{j=1}^{l_{\rm e}} (F_{e_{{\rm e},v,j}}) \times \{1\}$. The surface ${\sqcup}_{j=1}^{l_{\rm e}} (F_{e_{{\rm e},v,j}}) \times \{1\}$ is changed into one represented as the disjoint union of finitely many copies of surfaces each of which is diffeomorphic to either $S^2$ or $S^1 \times S^1$ and by attaching all $1$-handles corresponding to the critical points of index $1$ for the local Morse function disjointly, we have a surface diffeomorphic to ${\sqcup}_{j=1}^{l_{\rm d}} F_{e_{{\rm d},v,j}}$ and a union of connected components of the boundary of the local $3$-dimensional compact and connected manifold.
By attaching these $2$-handles one after another and after that attaching these $1$-handles one after another, we have a local simple Morse function on the local $3$-dimensional compact and connected manifold such that the regular level sets consist of connected components diffeomorphic to $S^2$, $S^1 \times S^1$, or $K^2$. 

In the case where ${\sqcup}_{j=1}^{l_{\rm e}} F_{e_{{\rm e},v,j}}$ contains no connected component diffeomorphic to $K^2$, we can discuss in the same way. 
% consider the identification with ${\sqcup}_{j=1}^{l_{\rm e}} (F_{e_{{\rm e},v,j}}$  $F_{e_{v_2,2}}) \times \{0\}$ by the map mapping $x$ to $(x,0)$ and by attaching ${\sqcup}_{j=1}^{l_{\rm e}} (F_{e_{{\rm e},v,j}}$
In addition, the Reeb data of a simple Morse function with regular level sets consisting of surfaces diffeomorphic to $S^2$, $S^1 \times S^1$, or $K^2$, are as presented in FIGURE \ref{fig:4}. 
 \begin{figure}
	\includegraphics[width=70mm,height=65mm]{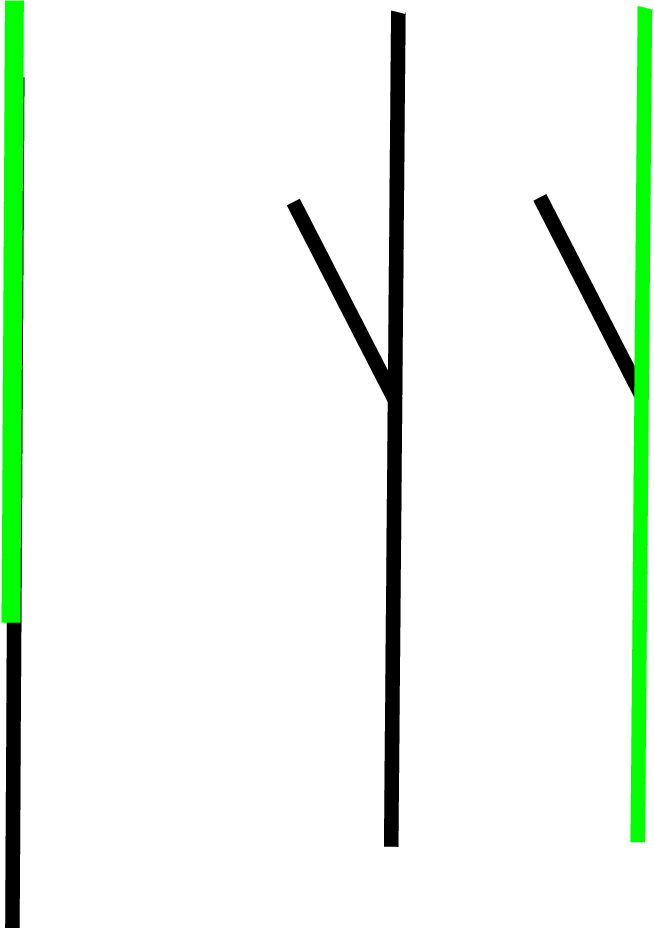}
	\caption{A local representation of the Reeb data of a simple Morse function with regular level sets consisting of surfaces diffeomorphic to $S^2$, $S^1 \times S^1$, or $K^2$. Note that for the black colored edges $e$, $F_e=S^2$ always holds, and that either the following holds in addition. For the two green colored edges $e$ here, $F_e=S^1 \times S^1$ always hold or $F_e=K^2$ always hold.}
	\label{fig:4}
\end{figure}
By fundamental arguments on handles in low ($2$- or $3$-)dimensional manifolds, we can deform the local simple Morse functions into new simple Morse functions as in FIGURE \ref{fig:5}, where their Reeb data are presented locally, only.
 \begin{figure}
	\includegraphics[width=70mm,height=65mm]{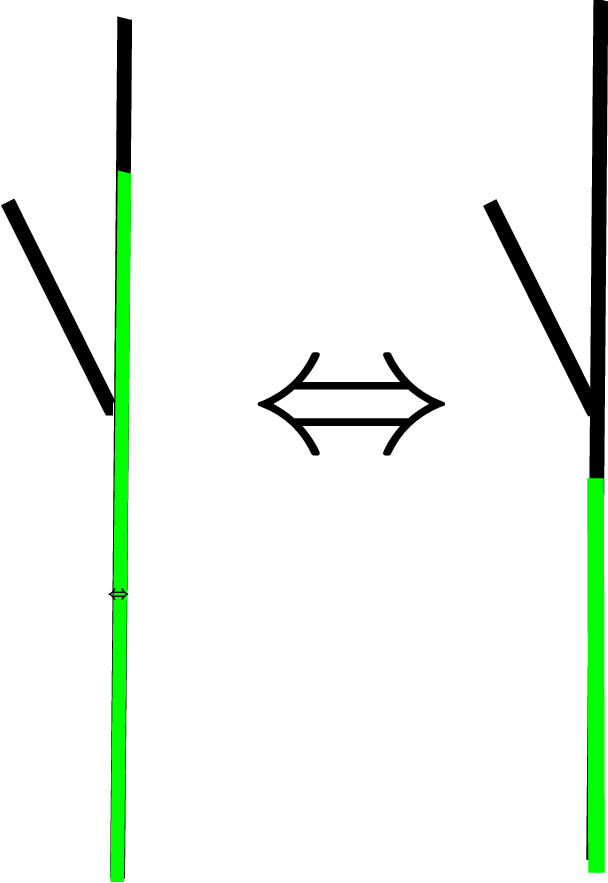}
	\caption{An important local deformation of simple Morse functions for FIGURE \ref{fig:4}.}
	\label{fig:5}
\end{figure}
We can see that the manifold $M$ is diffeomorphic to $S^3$ or represented as a connected sum of finitely many manifolds each of which is diffeomorphic to at least one of the following. 
Each of the manifolds is regarded as the manifold of the domain of one of the functions in Step 1 or the manifold of the domain of a simple Morse function such that the regular level sets are disjoint unions of spheres and that the the critical points of the function are of index $1$ or $m-1$ with $m=3$. Note that some result of \cite{saeki3} and related arguments, especially, \cite[Corollary 3.14 and Theorem 6.5]{saeki3} (or our Theorem \ref{thm:2}) and related arguments, are also essential.
\begin{itemize}
\item  $S^1 \times S^2$.
\item  Lens spaces.
\item A non-orientable manifold of degree $1$, which may be diffeomorphic to $S^1 \tilde{\times } S^2$, the non-trivial smooth bundle over $S^1$ whose fiber is diffeomorphic to $S^2$.
\end{itemize}
This completes the proof.
\end{proof}
Remark \ref{rem:1} is related to the proof above and important in the proof of Theorem \ref{thm:4} (\ref{thm:4.2}), presented later.
\begin{Rem}
\label{rem:1}
Related to Theorem \ref{thm:4} (\ref{thm:4.1}), we can have a Morse function $f_{\{v_j\}_{j=1}^5,\{e_j\}_{j=1}^4,\{F_{e_j}\}_{j=1}^4}:M:={{\mathbb{R}}P}^2 \times S^1 \rightarrow \mathbb{R}$ such that the Reeb digraph is a path digraph on $5$ vertices $v_j$ ($j=1,2,3,4,5$) with the edge set $\{e_1,e_2,e_3,e_4\}$ each element $e_j$ of which is oriented as a directed edge departing from $v_j$ and entering $v_{j+1}$ and with $F_{e_j}=S^2$ for $j=1,4$ and $F_{e_j}=K^2$ ($j=2,3$). Furthermore, by considering handles, we can have the function in such a way that the preimage ${q_{f_{\{v_j\}_{j=1}^5,\{e_j\}_{j=1}^4,\{F_{e_j}\}_{j=1}^4}}}^{-1}(v_i)$ contains exactly one critical point (two critical points) of $f_{\{v_j\}_{j=1}^5,\{e_j\}_{j=1}^4,\{F_{e_j}\}_{j=1}^4}$ for $i=1,2,4,5$ (resp. $i=3$). Note that in the case $j=3$, the two critical points of the function mapped to $v_3$ are of index $1$ and $2$, respectively. 
The 2nd Stiefel-Whitney class of ${\mathbb{R}P}^2$ does not vanish and from this, that of the manifold $M$ does not vanish. This function can be deformed to a simple Morse function as presented in \cite[Example 6 and FIGURE 1]{kitazawa2}. We avoid such regular level sets, containing surfaces diffeomorphic to ${\mathbb{R}P}^2$, in explicit and general situations of the present paper.
\end{Rem}
\begin{proof}[A proof of Theorem \ref{thm:4} (\ref{thm:4.2})]
"The "If" part" is same as "A proof of Theorem \ref{thm:4} (\ref{thm:4.1})".
From exposition in Remark \ref{rem:1}, as in "A proof of Theorem \ref{thm:4} (\ref{thm:4.1})", the given Morse function is deformed to a simple Morse function (locally and globally, similarly). This completes the proof.
\end{proof}
\section{Some classifications of Morse functions up to isomorphisms of ($m-1$)-labeled-digraphs with $m=2,3$.}
Classifying of Morse functions of certain classes is fundamental and important. It is natural to consider classifications up to {\it topological equivalence} and {\it $C^{\infty}$ equivalence}. In short, we consider the equivalence relation in the following for two smooth functions $c_1:X_1 \rightarrow \mathbb{R}$ and  $c_2:X_1 \rightarrow \mathbb{R}$: they are equivalent up to {\it topological} (resp. {\it $C^{\infty}$}) {\it equivalence} if there exists a pair $({\phi}_X:X_1 \rightarrow X_2,{\phi}_{\mathbb{R}}:\mathbb{R} \rightarrow \mathbb{R})$ of homeomorphisms (resp. {\it diffeomorphisms}) satisfying the relation ${\phi}_{\mathbb{R}} \circ f_1=f_2 \circ {\phi}_X$.  Such classifications have been done for simple Morse(-Bott) functions on closed surfaces, as shown in \cite{kulinich, lychakprishlyak, martinezalfaromezasarmientooliveira} for example. In short, they are classified essentially by the Reeb data, where such terminologies are used first by the author.

Classifications of functions up to isomorphisms of ($m-1$)-labeled-digraphs date back to Sharko's pioneering study \cite{sharko}. Sharko has considered reconstructing nice smooth functions with given Reeb digraphs on closed surfaces. He has constructed such functions which are locally Morse or represented by a certain elementary polynomial at critical points, being isolated in the case.  In \cite{masumotosaeki}, Masumoto and Saeki have extended this to arbitrary finite (di)graphs and constructed smooth functions on closed surfaces whose critical sets may not be isolated. Later, in \cite{michalak1}, Michalak has considered reconstruction of Morse functions such that level sets containing no critical point are spheres for digraphs as in Theorem \ref{thm:3}. In addition, Michalak has also presented the following theorem.
\begin{Thm}
\label{thm:5}
Suppose that a $1$-labeled-pre-digraph $(G,\{S^1\}_{e \in E_G})$ is given and that the following are satisfied.
\begin{itemize}
\item The $1$st Betti number of the graph $G$ is $a \geq 0$, where the orientation is forgotten. The graph $G$ has at least two edges.
\item The number of vertices of degree $2$ of $G$ is $b$.
\end{itemize}
For any closed, connected and orientable surface $M$ of genus $g \geq a+b$, we have a Morse function $f:M \rightarrow \mathbb{R}$ whose Reeb data are isomorphic to $(G,\{S^1\}_{e \in E_G})$. 
In addition, suppose that a closed, connected and orientable surface $M_0$ admits a Morse function $f_0:M_0 \rightarrow \mathbb{R}$ whose Reeb data are isomorphic to $(G,\{S^1\}_{e \in E_G})$, then $M_0$ must be of genus at least $a+b$. 
\end{Thm}
Note that the non-orientable case has been shown, where we omit. Later, Gelbukh has shown the Morse-Bott case in \cite{gelbukh1, gelbukh2}, where we omit.

Our new related result is Theorem \ref{thm:6}.
\begin{Thm}
\label{thm:6}
Suppose that a $2$-labeled-pre-digraph $(G,\{F_e\}_{e \in E_G})$ is given and that the following are satisfied.
\begin{itemize}
\item $F_e=S^2,S^1 \times S^1,K^2$.
\item The $1$st Betti number of the graph $G$ is $a_0 \geq 0$, where the orientation is forgotten. The graph $G$ has at least two edges.
\item The number of vertices $v$ of degree $2$ incident to exactly two edges $e_{v,1}$ and $e_{v,2}$ with $F_{e_{v,1}}$ and $F_{e_{v,2}}$ being spheres is $n_{(G,\{F_e\}_{e \in E_G}),S^2,S^2} \geq 0$.
\item For each vertex $v$, either an edge $e_{{\rm  e},v}$ entering $v$ with ${q_f}^{-1}(p_e)$ diffeomorphic to $K^2$ or an edge $e_{{\rm  d},v}$ departing from $v$ with ${q_f}^{-1}(p_e)$ diffeomorphic to $K^2$, exists, as in Theorem \ref{thm:4} {\rm (}\ref{thm:4.1}{\rm )}.
\item The number of edges $e$ with $F_{e}=S^1 \times S^1$ is $n_{(G,\{F_e\}_{e \in E_G}),S^1 \times S^1} \geq 0$ and the number of edges $e$ with $F_{e}=K^2$ is $n_{(G,\{F_e\}_{e \in E_G}),K^2} \geq 0$.
\end{itemize}
Given a $3$-dimensional closed, connected and orientable manifold $M$ diffeomorphic to a connected sum of the form $({\sharp}_{j=1}^{a_1} (S^1 \times S^2)) {\sharp} ({\sharp}_{j=1}^{a_2} (S^1 \tilde{\times} S^2)) \sharp ({\sharp}_{j=1}^{b} L_j) \sharp ({\sharp}_{j=1}^{n_{(G,\{F_e\}_{e \in E_G}),K^2}} K_j)$, where the notation are as follows{\rm :} $a_1 \geq 0$ and $a_2 \geq 0$ are integers satisfying $a_0+n_{(G,\{F_e\}_{e \in E_G}),S^2,S^2} \leq a_1+a_2$, $b$is an integer satisfying $b \leq n_{(G,\{F_e\}_{e \in E_G}),S^1 \times S^1}$, and each manifold $L_j$ and $K_j$ are a lens space and a non-orientable manifold of degree $1$ which may be diffeomorphic to $S^1 \tilde{\times} S^2$, respectively. For this, we have a Morse function $f:M \rightarrow \mathbb{R}$ whose Reeb data are isomorphic to  $(G,\{F_e\}_{e \in E_G})$. 
%In addition, suppose that a closed, connected and orientable surface $M_0$ admits a Morse function $f_0:M_0 \rightarrow \mathbb{R}$ whose Reeb data are isomorphic to $(G,\{F_e\}_{e \in E_G})$.

In addition, suppose that a closed, connected and orientable manifold $M_0$ admits a Morse function $f_0:M_0 \rightarrow \mathbb{R}$ whose Reeb data are isomorphic to  $(G,\{F_e\}_{e \in E_G})$. In this situation, $M_0$ must be of such manifolds, respecting these defined numbers, conversely. 

\end{Thm}
\begin{proof}
We first construct a desired Morse function on $M$ diffeomorphic to a connected sum of the form $({\sharp}_{j=1}^{a_1} (S^1 \times S^2)) {\sharp} ({\sharp}_{j=1}^{a_2} (S^1 \tilde{\times} S^2)) \sharp ({\sharp}_{j=1}^{b} L_j) \sharp ({\sharp}_{j=1}^{n_{(G,\{F_e\}_{e \in E_G}),K^2}} K_j)$.

It is essential to review Cases 2-1, 2-2, and 2-3 in "A proof of Theorem \ref{thm:4} (\ref{thm:4.1})". For Cases 2-1 and 2-2, the local Morse functions must be as presented there and for Case 2-2, we can make the 1st Betti number $l_{v_2}$ an arbitrary positive integer.
We present our desired construction for Case 2-3. We use an explicit fold map and compose it with the projection to the first component. This is presented in FIGUREs \ref{fig:6} and \ref{fig:7}.
We first construct a local fold map such that the restriction to its singular set is an embedding and that each singular point of the map is of degree $0$. 
We change suitably by adding connected components of the singular set of the new map to the previous map in such a way that each point of the new connected components is of degree $1$ and that the map is changed locally around the new connected components.

\begin{figure}
	\includegraphics[width=70mm,height=65mm]{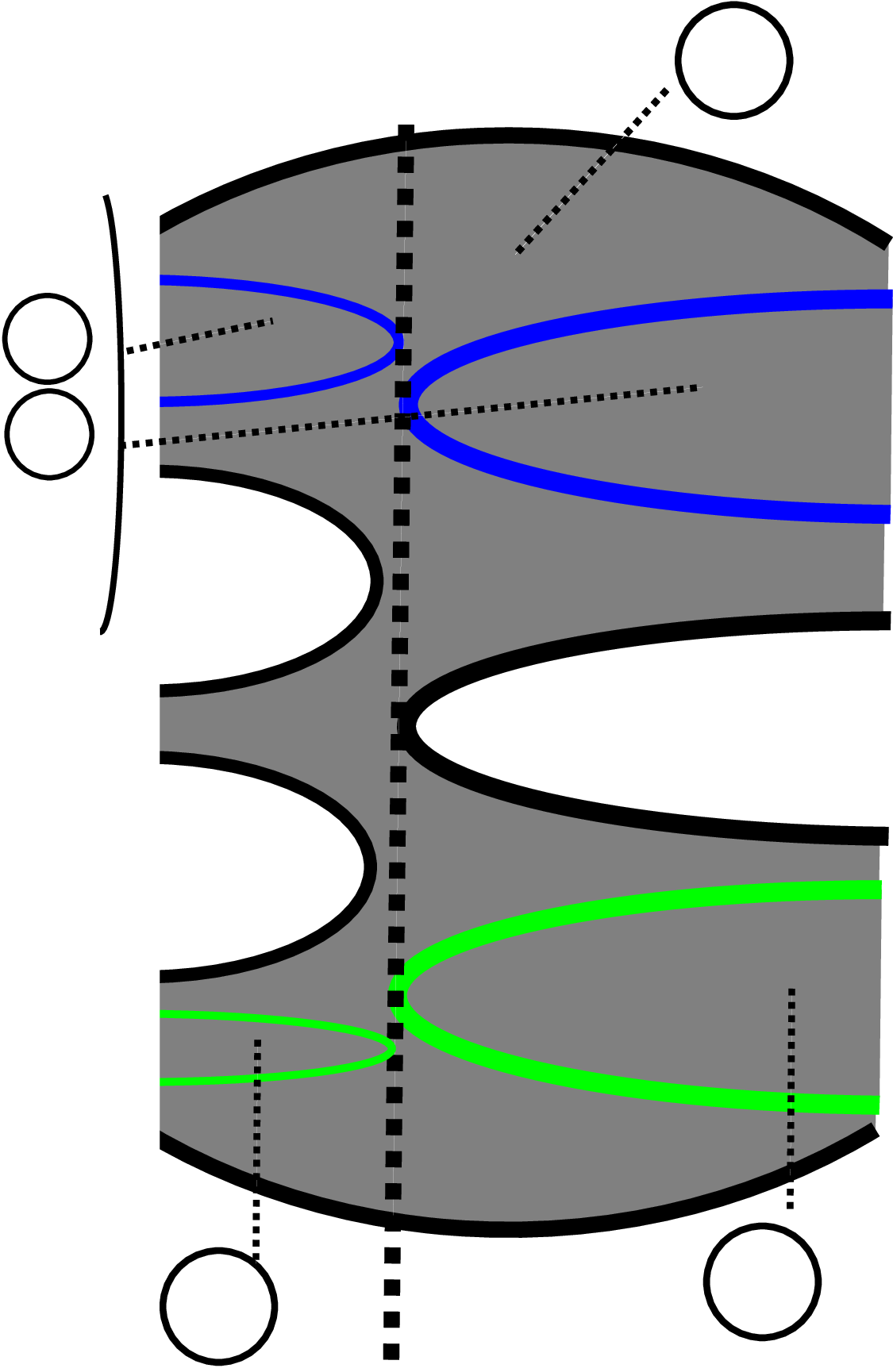}
	\caption{The image and the singular set of the local fold map into ${\mathbb{R}}^2$ and the preimages of some points by the local fold map respecting (a small neighborhood of) a vertex $v$ of Case 2-3. By composing the projection to the first component, we have a desired local Morse function. For the preimages of points by fold maps, see also \cite{saeki3}. There Saeki has established and explained so-called theory of ({\it singular}) {\it fibers}.}
	\label{fig:6}
\end{figure}
\begin{figure}
	\includegraphics[width=70mm,height=65mm]{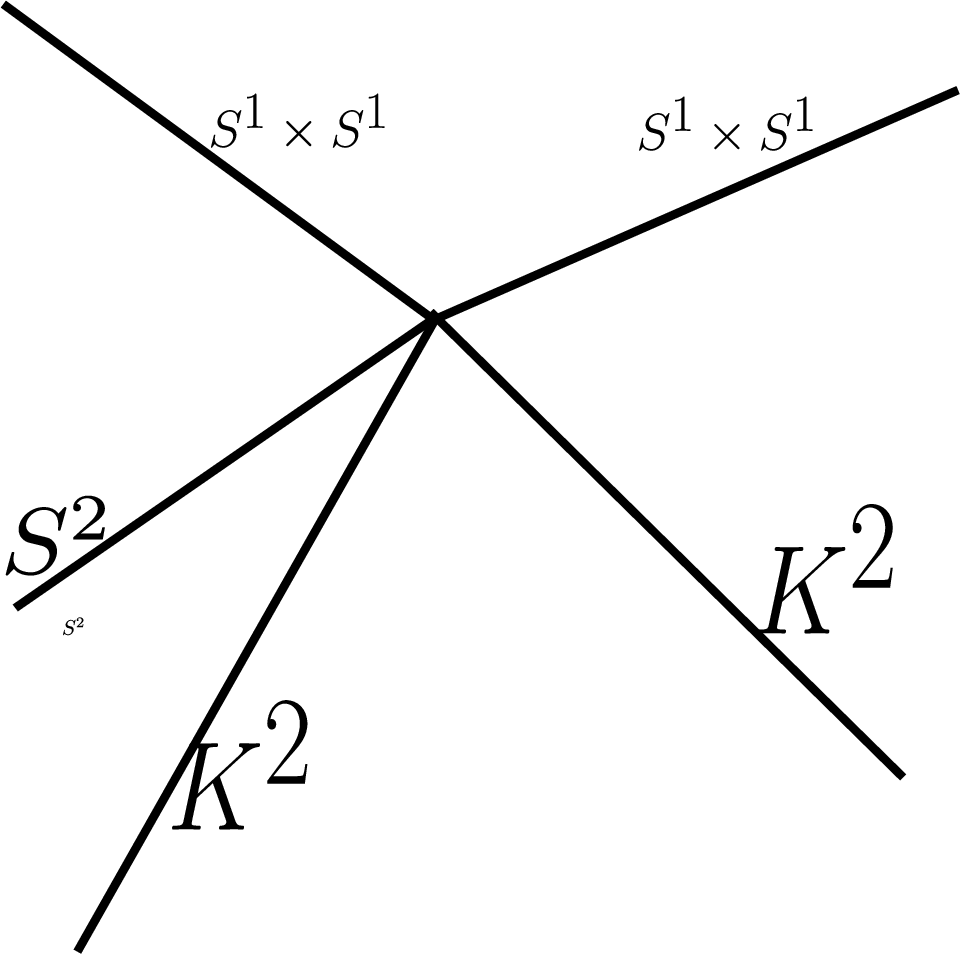}
	\caption{(Locally,) the Reeb data of the function which is, around the vertex $v$ and the preimage of the local fold map in FIGURE \ref{fig:6}, represented as the composition of the fold map with the canonical projection to the first component, is presented.}
	\label{fig:7}
\end{figure}
 See FIGUREs \ref{fig:6} and \ref{fig:7} again. For local structures of the fold maps, see \cite{saeki3} for example. There so-called theory of ({\it singular}) {\it fibers} is presented. A {\it singular fiber} of a smooth map means a ({\it germ} of) a smooth map around the preimage (by the map) containing some single point (of the map).
For this case, we can see that the local $3$-dimensional compact and connected manifold is obtained in the following way.

Suppose that ${\sqcup}_{j=1}^{l_{\rm e}} F_{e_{{\rm d},v,j}}$ contains no connected component diffeomorphic to $K^2$. In this case, for the surface ${\sqcup}_{j=1}^{l_{\rm e}} F_{e_{{\rm e},v,j}}$ we embed into ${\sqcup}_{j=1}^{l_{\rm e}} (F_{e_{{\rm e},v,j}}) \times \{t \mid 0 \leq t \leq 1\}$ by the map mapping $x$ to $(x,0)$ first. After that, first we attach $2$-handles corresponding to the critical points of index $2$ for the local Morse function disjointly to make each connected component of the surface ${\sqcup}_{j=1}^{l_{\rm e}} F_{e_{{\rm e},v,j}}$ to the sphere and second we attach $1$-handles corresponding to the critical points of index $1$ for the local Morse function disjointly to make the disjoint union of the spheres to a single sphere. After that, third, we attach $2$-handles corresponding to the critical points of index $2$ for the local Morse function disjointly to make a disjoint union of finitely many spheres and last, by attaching the remaining $1$-handles corresponding to the critical points of index $1$ for the local Morse function disjointly, the surface ${\sqcup}_{j=1}^{l_{\rm e}} (F_{e_{{\rm e},v,j}}) \times \{1\}$ is changed into one diffeomorphic to ${\sqcup}_{j=1}^{l_{\rm d}} F_{e_{{\rm d},v,j}}$ and a union of connected components of the boundary of the local $3$-dimensional compact and connected manifold. Corresponding to the handles, we can locally deform the function to a simple Morse function in such a way that the 1st Betti number does not increase locally. By the deformation as in FIGURE \ref{fig:5} and "A proof of Theorem \ref{thm:4} (\ref{thm:4.1})", we have a simple Morse function as in "A proof of Theorem \ref{thm:4} (\ref{thm:4.1})".
We can see that the manifold $M$ is a connected sum of the following manifolds.
\begin{itemize}
\item $a_1$ copies of $S^1 \times S^2$ and $a_2$ copies of $S^1 \tilde{\times} S^2$, where we use the integers $a_1 \geq 0$ and $a_2 \geq 0$ satisfying $a_0+n_{(G,\{F_e\}_{e \in E_G}),S^2,S^2} \leq a_1+a_2$. This is regarded as the $m$-dimensional manifold of the domain of a simple Morse function such that the regular level sets are disjoint unions of spheres and that the the critical points of the function are of index $1$ or $m-1$, where $m=3$.
For this, see \cite{saeki4} again. In addition, see also \cite{saekisuzuoka} and \cite{kitazawa01, kitazawa02, kitazawa03}, especially, \cite[Corollary 4.8]{saekisuzuoka} and \cite[Corollary 4]{kitazawa03}, which are on the isomorphisms between the fundamental groups of the manifolds of the domains and those of the Reeb graphs.
\item  At most $b \leq n_{(G,\{F_e\}_{e \in E_G}),S^1 \times S^1}$ lens spaces. This is for the manifolds $L_j$.
\item Exactly $n_{(G,\{F_e\}_{e \in E_G}),K^2}$ manifolds which are diffeomorphic to $S^1 \tilde{\times} S^2$ or more general non-orientable manifolds of degree $1$. This is for the manifolds $K_j$.

\end{itemize}
We can also see that the integers $a_1$, $a_2$ and  $b$ can be chosen in an arbitrary way under the constraint.
This completes the proof of the former part.

We explain the latter part, the topology (differentiable structure) of $M_0$. This can be understood by the present argument and the argument in "A proof of Theorem \ref{thm:4} (\ref{thm:4.1})".

This completes the proof.   
\end{proof}
Note that the case defined by the condition $F_e=S^2,S^1 \times S^1$ in Theorem \ref{thm:6} is a main theorem of \cite{kitazawa7} (\cite[Theorem 2]{kitazawa7}) and that the previous result is improved.
Note that related to this, contribution of the author to studies on nice smooth functions with given Reeb graphs is important. For example, the author has pioneered studies on reconstruction of nice smooth functions whose Reeb digraphs are isomorphic to given digraphs and whose level sets are given (($m-1$)-dimensional) smooth closed manifolds and presented the article \cite{kitazawa0}, followed by the author himself in {kitazawa1, kitazawa2} and the preprint \cite{kitazawa6}. We do not assume related knowledge or understanding. Note also that the notion of Reeb data is defined first in the present paper, formally.
\begin{Prob}
	\label{prob:2}
Can we consider similar theorems in the following case, for example.
\begin{enumerate}
\item The case of Theorem \ref{thm:2}. In the preprint \cite{kitazawa8}, partial or explicit investigations are presented.
\item The case of Theorem \ref{thm:4} (\ref{thm:4.2}). With a little effort, we may have an answer.
\end{enumerate}
\end{Prob}

 \section{Conflict of interest and Data availability.}
  \noindent {\bf Conflict of interest.} \\
 The author is a researcher at Osaka Central Advanced Mathematical Institute, or a member of OCAMI researchers. This is supported by MEXT Promotion of Distinctive Joint Research Center Program JPMXP0723833165. He thanks this opportunity. He is not employed there. \\
Related to this, the author has given talks related to \cite{kitazawa7, kitazawa8}, in the conference "Mathematical Theory of Knots VIII (https://sites.google.com/view/knots2025nuchs)"
 and "The 2026 Annual Meeting of The Mathematical Society of Japan (The Topology Division)", supported by several projects. For example, for the latter conference, the author has been supported by JSPS KAKENHI Grant Number JP23H05437 (Principal investigator: Osamu Saeki). The author would like to thank people related to them for supports and opportunities. \\
  %Some of works by other researchers and this version may overlap in some of the contents due to the nature that our problems are natural in theory of Morse functions and applications to differential topology and that related mathematical studies are very fundamental and classical in some senses, for example. However the present version of our paper is presented independent of these work. \\
  %Saga Souhatsu Mathematical Seminar (http://inasa.ms.saga-u.ac.jp/Japanese/saga-souhatsu.html), inviting the author as a speaker, is funded and supported by JST Fusion Oriented REsearch for disruptive Science and Technology JPMJFR202U: the author was a speaker on 2024/7/12 supported by this project.\\
  \ \\
  {\bf Data availability.} \\
 No other data is generated. We do not assume non-trivial arguments in preprints formally unpublished. Referring to these preprints to some extent is no problem.

\end{document}